\documentclass[10pt]{article}

\usepackage[T1]{fontenc}

\usepackage{amsmath}
\usepackage{amsthm}
\usepackage{amssymb}
\usepackage{latexsym}

\usepackage{enumitem}
\usepackage{calc}
\usepackage{xspace}

\usepackage{tikz}
\usetikzlibrary{calc,backgrounds,hobby}

\usepackage[small, pagestyles]{titlesec}
\usepackage[small, it]{caption}

\usepackage[square,comma,numbers,sort&compress]{natbib}
\usepackage[multiple, flushmargin, hang]{footmisc}

\usepackage{color}
\definecolor{lightgray}{rgb}{0.8,0.8,0.8}
\definecolor{darkgray}{rgb}{0.7,0.7,0.7}
\definecolor{darkblue}{rgb}{0,0,.4}

\usepackage[bookmarks]{hyperref}
\hypersetup{
	colorlinks=true,
	linkcolor=black,
	anchorcolor=black,
	citecolor=black,
	urlcolor=black,
	pdfpagemode=UseThumbs,
	pdftitle={Linear clique-width and modular decomposition},
	pdfsubject={Combinatorics},
	pdfauthor={Brignall, Opler, and Vatter},
}

\newcommand{\minisec}[1]{\medskip\noindent\textbf{#1.}}

\newcommand{\Av}{\operatorname{Av}}
\newcommand{\lcw}{\operatorname{lcw}}

\newcommand{\VH}{\;\!\!V\!(\!H\;\!\!)\!}

\newcommand{\plotpoints}[2][v]{%
	\foreach \i/\j [count=\k] in {#2} {%
		\node (#1\k) [fill,circle,inner sep=0pt,minimum size=5pt] at (\i,\j) {};
	}%
}
\newcommand{\drawjoin}[2]{%
	\foreach \nodea in {#1} {%
		\foreach \nodeb in {#2} {%
			\edef\tempA{\nodea}%
			\edef\tempB{\nodeb}%
			\ifx\tempA\tempB
			\else
				\draw (\nodea)--(\nodeb);
			\fi
		}%
	}%
}

\theoremstyle{plain}
\newtheorem{theorem}{Theorem}[section]
\newtheorem{proposition}[theorem]{Proposition}
\newtheorem*{proposition*}{Proposition}

\newtheorem{observation}[theorem]{Observation}

\newtheorem*{remark*}{Remark}

\renewenvironment{abstract}
  {
    \begin{list}{}%
      {\setlength{\rightmargin}{1in}%
       \setlength{\leftmargin}{1in}}%
    \item[]\ignorespaces\begin{small}
  }
  {
    \end{small}\unskip\end{list}
  }

\newpagestyle{main}[\small]{
  \headrule
  \sethead[\usepage][][]
          {\sc Linear Clique-Width and Modular Decomposition}{}{\usepage}
}

\titleformat{\section}
  {\large\sc}
  {\thesection.}{1em}{}

\title{\sc Linear Clique-Width and Modular Decomposition}

\author{
  Robert Brignall%
  \footnote{The Open University, School of Mathematics and Statistics, Milton Keynes, England UK.},
  \quad
  Michal Opler%
  \footnote{Czech Technical University, Department of Theoretical Computer Science, Prague, Czech Republic.},
  \quad
  and
  \quad
  Vincent Vatter%
  \footnote{University of Florida, Department of Mathematics, Gainesville, Florida USA.}
}

\date{\today}
\vspace{-0.25in}

\begin{document}
\maketitle

\begin{abstract}
A hereditary class of graphs has bounded clique-width if and only if its prime members do, but this lifting property fails for linear clique-width. We prove that a hereditary class has bounded linear clique-width if and only if its prime members do and it contains neither all quasi-threshold graphs nor all complements of quasi-threshold graphs. This generalizes a result of Brignall, Korpelainen, and Vatter, who established the result for cographs.
\end{abstract}

\pagestyle{main}

\section{Introduction}

A hereditary class of graphs has bounded clique-width if and only if its prime members have bounded clique-width; this follows easily from the definition. For linear clique-width, however, this lifting property fails. The class of cographs ($P_4$-free graphs) provides the simplest example: its only prime members, $K_2$ and $\overline{K}_2$, have linear clique-width bounded by~$2$, yet the class itself has unbounded linear clique-width. Brignall, Korpelainen, and Vatter~\cite{brignall:linear-clique-w:} identified the source of this failure within cographs, showing that the quasi-threshold graphs and their complements are the only obstructions.

\begin{theorem}[Brignall, Korpelainen, and Vatter~{\cite[Theorem 1.1]{brignall:linear-clique-w:}}]
\label{thm:bkv}
A hereditary class of cographs has bounded linear clique-width if and only if it contains neither all quasi-threshold graphs nor the complements of all quasi-threshold graphs.
\end{theorem}

Our main result shows that this is not special to cographs: quasi-threshold graphs and their complements are, in general, the only obstructions to lifting bounded linear clique-width from the prime members to the whole class.

\begin{theorem}
\label{thm:primes:bdd:lcw}
A hereditary class of graphs has bounded linear clique-width if and only if its prime members have bounded linear clique-width and it contains neither all quasi-threshold graphs nor all complements of quasi-threshold graphs.
\end{theorem}

There has been considerable recent interest in identifying minimal hereditary classes of unbounded clique-width. Lozin~\cite{lozin:minimal-classes:} gave two examples of such classes, Collins, Foniok, Korpelainen, Lozin, and Zamaraev~\cite{collins:infinitely-many:} found infinitely many more, and Brignall and Cocks~\cite{brignall:uncountably-man:cw} constructed uncountably many. Atminas, Brignall, Lozin, and Stacho~\cite{atminas:minimal-classes:} showed that there exist minimal classes defined by finitely many forbidden induced subgraphs, disproving a conjecture of Daligault, Rao, and Thomass\'e~\cite{daligault:well-quasi-orde:}, and in a separate work, Brignall and Cocks~\cite{brignall:a-framework-for:} developed a unifying framework encompassing nearly all known examples. As observed by Alecu, Kant\'e, Lozin, and Zamaraev~\cite{alecu:between-clique-:}, many of these minimal classes are also minimal for linear clique-width, and this phenomenon recurs throughout the framework of~\cite{brignall:a-framework-for:}. Theorem~\ref{thm:primes:bdd:lcw} complements this line of work: once one checks that a hereditary class excludes some quasi-threshold graph and some co-quasi-threshold graph, the question of whether the class has bounded linear clique-width reduces entirely to its prime members.

The proof of Theorem~\ref{thm:bkv} in~\cite{brignall:linear-clique-w:} followed the blueprint of Albert, Atkinson, and Vatter~\cite{albert:subclasses-of-t:}, who proved an analogous enumerative result for separable permutations. That blueprint involves a substantial amount of well-quasi-order machinery which, while necessary to obtain the enumerative results in the permutation setting, is not needed for the graph-theoretic result. Indeed, this reliance on well-quasi-ordering was noted by Alecu, Kant\'e, Lozin, and Zamaraev~\cite{alecu:between-clique-:}, who commented that it limited the applicability of the approach. Our proof of Theorem~\ref{thm:primes:bdd:lcw} is considerably more direct, requiring only modular decomposition and explicit universal quasi-threshold graphs, and avoiding well-quasi-ordering entirely.

\section{Linear Clique-Width and Universal Quasi-Threshold Graphs}
\label{sec:Qt}

We write $G\uplus H$ for the \emph{disjoint union} of graphs $G$ and $H$, and $G\ast H$ for their \emph{join}, formed from $G\uplus H$ by adding all edges with one endpoint in $V(G)$ and the other in $V(H)$. We write $H\le G$ to mean that $H$ is isomorphic to an induced subgraph of $G$. A set of graphs closed under isomorphism and closed downward under the induced subgraph order is called a \emph{hereditary class}, or simply a \emph{class}.

The \emph{linear clique-width} of a graph $G$, denoted $\lcw(G)$, is the smallest number of labels needed to construct $G$ by a sequence of the following three operations:
\begin{itemize}[noitemsep, topsep=0pt]
\item introduce a new vertex and assign it a label,
\item for labels $i \neq j$, add edges between all vertices labeled $i$ and all vertices labeled~$j$, and
\item for labels $i \neq j$, change the label of all vertices labeled $i$ to $j$.
\end{itemize}
Such a sequence is called an \emph{lcw expression} for $G$, and an lcw expression is \emph{efficient} if it uses exactly $\lcw(G)$ labels. A label in an lcw expression is a \emph{sink} if it never appears in an edge-creation or relabeling operation; vertices may be assigned to a sink label, but it is otherwise inert.

Linear clique-width is a sequential restriction of \emph{clique-width}. Clique-width was introduced by Courcelle, Engelfriet, and Rozenberg~\cite{courcelle:handle-rewritin:}, and it differs from linear clique-width in that clique-width expressions may additionally take the disjoint union of two previously constructed labeled graphs. A closely related parameter, \emph{NLC-width}, was introduced by Wanke~\cite{wanke:k-nlc-graphs-an:}, and Gurski and Wanke~\cite{gurski:on-the-relation:} were the first to study the linear variant of this family of parameters, under the name \emph{linear NLC-width}. Linear clique-width was subsequently introduced by Lozin and Rautenbach~\cite{lozin:the-relative-cl:}.

We consider the following families of graphs. A \emph{quasi-threshold graph} (also called a \emph{trivially perfect graph}) is any graph that can be constructed from $K_1$ by repeatedly taking joins with $K_1$, and disjoint unions with other quasi-threshold graphs; equivalently, quasi-threshold graphs are the $\{P_4,C_4\}$-free graphs. A \emph{co-quasi-threshold graph} is the complement of a quasi-threshold graph. A \emph{cograph} is any graph that can be constructed from $K_1$ by repeatedly taking disjoint unions and joins; equivalently, cographs are the $P_4$-free graphs. Thus, the class of quasi-threshold graphs forms a proper subclass of the class of cographs.

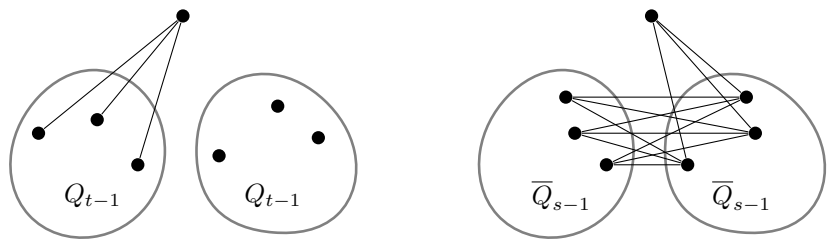
\begin{figure}
	\begin{center}
\begin{tikzpicture}[scale=0.6, baseline=(current bounding box.center)]
	\plotpoints[a]{0/3}
	\begin{scope}[shift={(-2,0)}]
		\plotpoints[l]{-1.2/0.4, 0.1/.7, 1.0/-0.3}
		\node[] at (0,-1) {$Q_{t-1}$};
		\begin{scope}[on background layer]
		\draw[use Hobby shortcut, closed, line width=1pt, opacity=0.5]
		(0:1.6)..(45:1.7)..(90:1.8)..(135:1.7)..(180:1.8)..(225:2.0)..(270:1.9)..(315:1.7);
		\end{scope}
	\end{scope}
	\begin{scope}[shift={(2,0)}]
		\plotpoints[r]{-1.2/-0.1, 0.1/1.0, 1.0/0.3}
		\node[] at (0,-1) {$Q_{t-1}$};
		\begin{scope}[on background layer]
		\draw[use Hobby shortcut, closed, line width=1pt, opacity=0.5]
		(0:1.8)..(45:1.6)..(90:1.7)..(135:1.8)..(180:1.7)..(225:1.8)..(270:1.8)..(315:2.0);
		\end{scope}
	\end{scope}
	\drawjoin{a1}{l1,l2,l3}
\end{tikzpicture}
	\qquad\qquad
\begin{tikzpicture}[scale=0.6, baseline=(current bounding box.center)]
	\plotpoints[a]{0/3}
	\begin{scope}[shift={(-2,0)}]
		\plotpoints[l]{0.3/0.4, 0.1/1.2, 1.0/-0.3}
		\node[] at (0,-1) {$\overline{Q}_{s-1}$};
		\begin{scope}[on background layer]
		\draw[use Hobby shortcut, closed, line width=1pt, opacity=0.5]
		(0:1.6)..(45:1.7)..(90:1.8)..(135:1.7)..(180:1.8)..(225:2.0)..(270:1.9)..(315:1.7);
		\end{scope}
	\end{scope}
	\begin{scope}[shift={(2,0)}]
		\plotpoints[r]{-1.2/-0.3, 0.1/1.2, 0.3/0.4}
		\node[] at (0,-1) {$\overline{Q}_{s-1}$};
		\begin{scope}[on background layer]
		\draw[use Hobby shortcut, closed, line width=1pt, opacity=0.5]
		(0:1.8)..(45:1.6)..(90:1.7)..(135:1.8)..(180:1.7)..(225:1.8)..(270:1.8)..(315:2.0);
		\end{scope}
	\end{scope}
	\drawjoin{a1}{r1,r2,r3}
	\drawjoin{l1,l2,l3}{r1,r2,r3}
\end{tikzpicture}
	\end{center}
\caption{The construction process for the graphs $Q_t$ and $\overline{Q}_s$ on the left and right, respectively.}
\label{fig:ext1}
\end{figure} 

We now introduce graphs that are universal for the classes of quasi-threshold and co-quasi-threshold graphs. Our \emph{universal quasi-threshold graphs} $Q_1, Q_2, \dotsc$ are defined by $Q_1=K_1$ and, for $t\ge 2$,
\[
	Q_t = \bigl(K_1\ast Q_{t-1}\bigr)\uplus Q_{t-1}.
\]
Their complements $\overline{Q}_1, \overline{Q}_2, \dotsc$ satisfy $\overline{Q}_1 = K_1$ and, for $s\ge 2$,
\[
	\overline{Q}_s
	= \bigl(K_1\uplus \overline{Q}_{s-1}\bigr)\ast \overline{Q}_{s-1}.
\]
These graphs are depicted in Figure~\ref{fig:ext1}.

\begin{observation}
\label{obs:universal}
For every quasi-threshold graph $G$ there exists $t$ such that $G\le Q_t$, and for every co-quasi-threshold graph $G$ there exists $s$ such that $G\le \overline{Q}_s$.
\end{observation}
\begin{proof}
We prove the first claim by induction on $|V(G)|$; the second follows by complementation. The base case is immediate since $Q_1=K_1$. Every quasi-threshold graph on two or more vertices is either a disjoint union $G_1\uplus G_2$ of smaller quasi-threshold graphs, or the join $K_1\ast G'$ of a single vertex with a smaller quasi-threshold graph. In the first case, $G_1, G_2 \le Q_{t-1}$ by induction, so $G\le Q_{t-1}\uplus Q_{t-1}\le Q_t$. In the second, $G'\le Q_{t-1}$, so $G\le K_1\ast Q_{t-1}\le Q_t$.
\end{proof}

Every cograph has clique-width at most $2$ (this is immediate from the definitions and first observed by Courcelle and Olariu~\cite{courcelle:upper-bounds-to:}), but Gurski and Wanke~\cite{gurski:on-the-relation:} showed that this class has unbounded linear clique-width. Brignall, Korpelainen, and Vatter~\cite{brignall:linear-clique-w:} observed that linear clique-width is unbounded even for the subclass of quasi-threshold graphs. Combined with Observation~\ref{obs:universal}, this implies that the linear clique-width of the graphs $Q_t$ and $\overline{Q}_s$ is unbounded.

\section{Modular Decomposition}
\label{sec-mod-decomp}

We begin by recalling the theory of modular decomposition, which has been independently discovered in numerous contexts; see Gallai~\cite{gallai:transitiv-orien:,gallai:a-translation-o:} for the foundational work in the graph setting. A \emph{module} of a graph $G$ is a subset $M\subseteq V(G)$ with the property that every vertex outside $M$ is either adjacent to all of the vertices of $M$ or to none of them. The empty set, each singleton, and $V(G)$ are all modules; we call these \emph{trivial}. A graph on at least two vertices is \emph{prime} if all of its modules are trivial. Note that there are no prime graphs on three vertices, and that $K_2$ and $\overline{K}_2$ are both prime under our definition; some other works require prime graphs to have at least four vertices.

Given a graph $G$, we say that $G$ is an \emph{inflation} of the graph $H$ if $G$ can be obtained by replacing every vertex $v \in V(H)$ with a nonempty graph $G_v$, so that every vertex of $G_u$ is adjacent to every vertex of $G_v$ in $G$ if and only if $u$ is adjacent to $v$ in $H$. We write this as $G = H[G_v : v \in V(H)]$, and call $H$ the \emph{skeleton} of the decomposition.

\begin{theorem}
\label{thm:mod:decomp}
For every graph $G$ on two or more vertices, there exists a unique prime graph $H$ on two or more vertices and nonempty graphs $\{G_v : v\in V(H)\}$ such that $G = H[G_v : v \in V(H)]$. Moreover, if $H$ has at least four vertices, then the graphs $G_v$ are also unique.
\end{theorem}

When the skeleton provided by Theorem~\ref{thm:mod:decomp} has four or more vertices, it is neither complete nor anti-complete, and the decomposition is unambiguous. When the skeleton is $\overline{K}_2$ or $K_2$, this means that either $G$ or its complement, respectively, is disconnected. In this case, we refine by decomposing into connected components or co-components (connected components of the complement), and take the skeleton to be an anti-complete or complete graph that is as large as possible. We refer to the result of this process as the first stage of the \emph{modular decomposition} of~$G$; note that the full modular decomposition would continue recursively to the modules, but we only need one stage at a time here.

The remainder of this section establishes several results relating linear clique-width to modular decomposition. We note that these results hold for arbitrary inflations, not only the unique inflations of prime graphs guaranteed by Theorem~\ref{thm:mod:decomp}. The first has several proofs in the literature, including in Brignall, Korpelainen, and Vatter~{\cite[Proposition 3.1]{brignall:linear-clique-w:}}; we provide another as it motivates the proofs that follow.

\begin{proposition}
\label{prop:lcw:inflation}
If $G=H[\{G_v : v\in V(H)\}]$, then
\[
	\lcw(G)
	\le
	\lcw(H) + \max_{v\in \VH} \lcw(G_v).
\]
\end{proposition}
\begin{proof}
Fix an efficient lcw expression for $H$. We build an lcw expression for $G$ by following the expression for $H$ step by step, replacing each vertex insertion with a construction of the entire corresponding module: when the expression for $H$ calls for the insertion of vertex $v$ with label $\ell$, we instead build a copy of $G_v$ using a set of $\lcw(G_v)$ labels reserved for module constructions and disjoint from those of the expression for $H$, then relabel every vertex of this copy to $\ell$, and continue. Since the same reserved labels are reused for each module, the resulting expression uses at most $\lcw(H)+\max_{v\in V(H)}\lcw(G_v)$ labels.
\end{proof}

When we are inflating a complete or anti-complete graph, we can slightly improve upon the bound of Proposition~\ref{prop:lcw:inflation}.

\begin{proposition}
\label{prop:lcw:inflation:comp}
If $G=H[\{G_v : v\in V(H)\}]$ where $H$ is complete or anti-complete, then
\[
	\lcw(G)\le 1+\max_{v\in \VH} \lcw(G_v).
\]
\end{proposition}
\begin{proof}
Note that when $H$ is anti-complete, this follows from Proposition~\ref{prop:lcw:inflation} since $\lcw(H)=1$, but we give a uniform argument covering both cases. We reserve one label, $\ell$, as a \emph{sink}: a label that is never relabeled and never used in any edge-creation step except as described below. We then build the modules $G_v$ one at a time, each using an efficient expression with $\lcw(G_v)$ labels, all distinct from $\ell$.

If $H$ is complete, then after completing a module $G_v$, we join each of its labels to $\ell$ (making its vertices adjacent to the vertices of all previously completed modules) and then relabel all of its vertices to $\ell$. If $H$ is anti-complete, we simply relabel all vertices of $G_v$ to $\ell$ after completing it, without any join steps. In both cases, the resulting expression uses at most $1+\max_{v\in V(H)}\lcw(G_v)$ labels.
\end{proof}

The following observation, whose proof amounts to nothing more than giving a single vertex its own private label, allows us to reorder the construction of a graph at the cost of that additional label. We use it in the subsequent proposition, where we are interested in bounding the linear clique-width of the two modules whose linear clique-width is largest.

\begin{proposition}
\label{prop:lcw:reorder}
For every graph $G$ and every vertex $v\in V(G)$, there is an lcw expression of $G$ with at most $\lcw(G)+1$ labels that begins by inserting the vertex $v$.
\end{proposition}
\begin{proof}
Fix an efficient lcw expression for $G$ and let $\ell$ be a new label not used in that expression. We build $G$ by first inserting $v$ with label $\ell$, and then following the original expression step by step with the following modifications. We skip the insertion of $v$ when it arises. Whenever the original expression relabels the label that $v$ would hold at that point, we do not apply this relabeling to $v$ (it retains the label $\ell$ throughout). Whenever the original expression creates edges between the label that $v$ would hold at that point and some other label, we additionally create edges between $\ell$ and that label. This produces the same graph using at most one additional label, namely, $\ell$.
\end{proof}

Combining Propositions~\ref{prop:lcw:inflation} and~\ref{prop:lcw:reorder}, we obtain the following result.

\begin{proposition}
\label{prop:lcw:two:modules}
Suppose that $G=H[G_v : v\in V(H)]$ where $H$ has at least two vertices. Then either there is a vertex $x$ of $H$ for which
\[
	\lcw(G_x)
	=
	\lcw(G),
\]
or there are distinct vertices $x$ and $y$ of $H$ for which
\[
	\lcw(G_x),\ \lcw(G_y)
	\ge
	\lcw(G)-\lcw(H)-1.
\]
\end{proposition}
\begin{proof}
Choose $x\in V(H)$ to maximize $\lcw(G_x)$. By Proposition~\ref{prop:lcw:reorder}, there is an expression for $H$ using at most $\lcw(H)+1$ labels in which the vertex~$x$ is inserted first. Following the proof of Proposition~\ref{prop:lcw:inflation}, we use this expression to build $G$: we first construct $G_x$ using $\lcw(G_x)$ labels, then relabel all of its vertices to the label assigned to $x$ in the expression for $H$, and continue following the expression for $H$, building each subsequent module using the same $\lcw(G_x)$ labels. The construction of $G_x$ uses $\lcw(G_x)$ labels, while the remainder of the construction uses at most
\[
	1+\lcw(H)+\max_{v\neq x} \lcw(G_v)
\]
labels. Since this is an lcw expression for $G$, one of these two quantities must be at least $\lcw(G)$. If the first is, then we have $\lcw(G_x)=\lcw(G)$. Otherwise, choosing $y\in V(H)\setminus\{x\}$ to maximize $\lcw(G_y)$, we obtain $\lcw(G_y)\ge\lcw(G)-\lcw(H)-1$, and $\lcw(G_x)\ge\lcw(G_y)$ by our choice of~$x$.
\end{proof}

\section{Proof of Theorem~\ref{thm:primes:bdd:lcw}}

Our main result follows from the following result, which can be considered as a more precise version of Observation~\ref{obs:universal}. Indeed, Proposition~\ref{prop:main} shows that any graph in a hereditary class satisfying the hypotheses of Theorem~\ref{thm:primes:bdd:lcw} has linear clique-width bounded by a function of $m$, $t$, and $s$, where $m$ bounds the linear clique-width of the prime members, and $Q_t$ and $\overline{Q}_s$ are the first universal quasi-threshold and co-quasi-threshold graphs not in the class. Conversely, if a hereditary class contains all quasi-threshold graphs, then by Observation~\ref{obs:universal} it contains all $Q_t$, so its linear clique-width is unbounded, and likewise if it contains all co-quasi-threshold graphs.

The idea of the proof is as follows. Given a graph $G$ of large linear clique-width, the modular decomposition produces a skeleton $H$ (of bounded linear clique-width) and modules whose linear clique-width is governed by Proposition~\ref{prop:lcw:two:modules}: that is, either one module matches the linear clique-width of $G$ (in which case we apply induction) or two of them have large linear clique-width. If $H$ is prime, there exist vertices with mixed adjacencies to these large modules, which directly yield the joins and disjoint unions needed to build $Q_t$ or $\overline{Q}_s$. If instead $H$ is complete or anti-complete, then $G$ provides only one of these two operations. We then apply the modular decomposition a second time to the module of largest linear clique-width, which is forced to have the ``opposite'' type of skeleton, and find what we need at this second level.

\begin{proposition}
\label{prop:main}
Let $G$ be a graph for which the linear clique-width of every prime induced subgraph is at most $m$, and let $t,s\ge 1$. If
\[
	\lcw(G)\ge (m+2)(t+s),
\]
then $G$ contains $Q_t$ or $\overline{Q}_s$ as an induced subgraph.
\end{proposition}

\begin{proof}
We proceed by induction on the number of vertices of $G$. We may assume that $m\ge 2$, as otherwise $G$ cannot contain $K_2$, so $G$ is anti-complete and cannot satisfy the bound in the proposition for any $t,s\ge 1$. In particular, $G$ has at least two vertices, establishing the base case of our induction.

We may also assume that $t,s\ge 3$. Cases with $t=1$ or $s=1$ are trivial, since $Q_1=\overline{Q}_1=K_1$ is an induced subgraph of every nonempty graph. For $t=2$, note that a graph not containing $Q_2=K_2\uplus K_1$ as an induced subgraph is complete multipartite, with linear clique-width at most~$2$. Since $m\ge 2$ and $s\ge 1$, we have $(m+2)(2+s)\ge 12>2$, so every graph satisfying the bound contains $Q_2\le Q_t$. Dually, a graph not containing $\overline{Q}_2=P_3$ is a disjoint union of cliques, with linear clique-width at most~$3$. Since $m\ge 2$ and $t\ge 1$, we have $(m+2)(t+2)\ge 12>3$, so every graph satisfying our bound also contains $\overline{Q}_2\le\overline{Q}_s$. We make this reduction because $t,s\ge 3$ ensures that~$Q_{t-1}$ is disconnected and $\overline{Q}_{s-1}$ is connected, properties we use in the argument below.

Now suppose that $G$ satisfies the hypotheses of the proposition with $m\ge 2$ and $t,s\ge 3$ and that the result holds for all graphs with fewer vertices. Apply the first stage of the modular decomposition to write
\[
	G=H[G_v : v\in V(H)]
\]
where $H$ is complete, anti-complete, or prime on at least four vertices, and in the first two cases the modules $G_v$ are the co-components or connected components of $G$, respectively. Let $x$ be a vertex of $H$ for which $\lcw(G_x)$ is maximal amongst all modules $G_v$. If $\lcw(G_x)=\lcw(G)$, then we are done by induction because $G_x$ itself satisfies the hypotheses and therefore contains either $Q_t$ or $\overline{Q}_s$.

Otherwise, Proposition~\ref{prop:lcw:two:modules} provides a vertex $y\neq x$ of $H$ such that
\[
	\lcw(G_x),\ \lcw(G_y)
	\ge
	\lcw(G)-\lcw(H)-1.
\]
Since $\lcw(H)\le m$ (as $H$ is an induced subgraph of $G$ that is either complete, anti-complete, or prime), we have
\[
	\lcw(G_x),\ \lcw(G_y)
	\ge
	\lcw(G)-m-1
	\ge
	(m+2)(t+s-1)+1.
\]
Since $G_x$ and $G_y$ have fewer vertices than $G$, the inductive hypothesis applies to them for all values of $t$ and $s$. Applying it with $(t-1,s)$ and $(t,s-1)$, we conclude that each of $G_x$ and $G_y$ contains~$Q_{t-1}$ or $\overline{Q}_s$, and each also contains $Q_t$ or $\overline{Q}_{s-1}$. If either contains $Q_t$ or $\overline{Q}_s$, then so does $G$ and we are done. We may therefore assume that both $G_x$ and~$G_y$ contain $Q_{t-1}$ and $\overline{Q}_{s-1}$.

First suppose that $H$ is prime on four or more vertices. Because $\{x,y\}$ does not form a module in~$H$, there is some vertex $z\in V(H)$ adjacent to precisely one of $x$ or $y$. If $x$ and $y$ are adjacent in~$H$, then~$G$ contains $\overline{Q}_s=(K_1\uplus \overline{Q}_{s-1})\ast\overline{Q}_{s-1}$, while if $x$ and $y$ are nonadjacent in $H$, then~$G$ contains $Q_t=(K_1\ast Q_{t-1})\uplus Q_{t-1}$, and in either case we are done.

It remains to consider the case where $H$ is complete or anti-complete. We treat the anti-complete case first, and then explain the modifications needed for the complete case.

Suppose that $H$ is anti-complete, so $G$ is the disjoint union of its connected components, two of which are $G_x$ and $G_y$. By Proposition~\ref{prop:lcw:inflation:comp}, $\lcw(G_x)\ge\lcw(G)-1\ge (m+2)(t+s)-1$. Since $G_x$ is connected, its skeleton is either complete or prime on at least four vertices.

If $G_x$ is the inflation of a complete graph (a join of co-components), then since $Q_{t-1}$ is disconnected, it lies entirely within a single co-component of $G_x$, and any vertex from another co-component is adjacent to all of $Q_{t-1}$. Thus $G_x$ contains $K_1\ast Q_{t-1}$, which together with the copy of $Q_{t-1}$ in $G_y$ gives $Q_t=(K_1\ast Q_{t-1})\uplus Q_{t-1}$ in $G$.

If $G_x$ is the inflation of a prime graph $H'$ on four or more vertices, we apply Proposition~\ref{prop:lcw:two:modules} to $G_x$. Let $a$ be a vertex of $H'$ maximizing $\lcw(G'_a)$. If $\lcw(G'_a)=\lcw(G_x)$, then $G'_a$ has fewer vertices than $G$ and satisfies the hypotheses of the proposition, so we are done by induction. Otherwise, Proposition~\ref{prop:lcw:two:modules} provides a vertex $b\neq a$ of $H'$ with
\[
	\lcw(G'_a),\ \lcw(G'_b)
	\ge
	\lcw(G_x)-m-1
	\ge
	(m+2)(t+s-1).
\]
Applying the inductive hypothesis to $G'_a$ and $G'_b$ as we did with $G_x$ and $G_y$ earlier, we may assume that they each contain $Q_{t-1}$ and $\overline{Q}_{s-1}$.

Since $H'$ is prime, $\{a,b\}$ is not a module, so there exists a vertex $c\in V(H')$ adjacent to precisely one of $a$ or $b$. Say $c$ is adjacent to $a$. Any vertex of $G'_c$ is then adjacent to all of $Q_{t-1}$ inside $G'_a$, giving a copy of $K_1\ast Q_{t-1}$ inside $G_x$. Combined with the copy of $Q_{t-1}$ in $G_y$ (a separate component of $G$), we obtain $Q_t=(K_1\ast Q_{t-1})\uplus Q_{t-1}$ in $G$, completing the proof in this case.

Now suppose that $H$ is complete, so $G$ is the join of its co-components, two of which are $G_x$ and $G_y$. By Proposition~\ref{prop:lcw:inflation:comp}, $\lcw(G_x)\ge (m+2)(t+s)-1$ as before. Since $G_x$ is co-connected, its skeleton is either anti-complete or prime on at least four vertices. If $G_x$ is the inflation of an anti-complete graph (a disjoint union of connected components), then $\overline{Q}_{s-1}$ is connected (since $s\ge 3$), so it lies within a single component of $G_x$, and any vertex from another component is nonadjacent to all of~$\overline{Q}_{s-1}$. Thus $G_x$ contains $K_1\uplus\overline{Q}_{s-1}$, which together with the copy of~$\overline{Q}_{s-1}$ in~$G_y$ (joined to~$G_x$ since~$H$ is complete) gives $\overline{Q}_s=(K_1\uplus\overline{Q}_{s-1})\ast\overline{Q}_{s-1}$ in~$G$. If instead $G_x$ has a prime skeleton, we find two large modules inside~$G_x$ each containing~$Q_{t-1}$ and~$\overline{Q}_{s-1}$, exactly as in the anti-complete case. The prime skeleton of $G_x$ again provides a vertex $c$ adjacent to precisely one of them, say $c$ is adjacent to $a$ but not to $b$. A vertex of $G'_c$ is then nonadjacent to all of $\overline{Q}_{s-1}$ inside $G'_b$, giving $K_1\uplus\overline{Q}_{s-1}$ inside $G_x$. Together with the copy of $\overline{Q}_{s-1}$ in $G_y$ (joined to $G_x$ since $H$ is complete), we obtain $\overline{Q}_s=(K_1\uplus\overline{Q}_{s-1})\ast\overline{Q}_{s-1}$ in~$G$. This completes the proof.
\end{proof}

We do not claim that the bound $(m+2)(t+s)$ in Proposition~\ref{prop:main} is optimal, although there is no slack in the induction. The linear dependence on $t+s$ is correct, since $\lcw(Q_t)$ itself grows linearly in $t$.

\section{Concluding Remarks}

Theorem~\ref{thm:primes:bdd:lcw} shows that the quasi-threshold graphs and their complements are the only obstructions to lifting bounded linear clique-width from the prime members to the whole class. Ferguson and Vatter~\cite[Theorem~5.1]{ferguson:letter-graphs-a:} proved an analogous result for lettericity, another graph width parameter: when the prime members of a class have bounded lettericity, the class has bounded lettericity if and only if it excludes all matchings, all co-matchings, and all ``stacked paths''.

More broadly, both results are instances of a familiar theme across combinatorics: if the ``atoms'' of a decomposition are controlled, then unbounded complexity can only arise from a self-embedding mechanism. The universal quasi-threshold graphs exhibit this concretely, as each~$Q_t$ contains two induced copies of~$Q_{t-1}$ with one of the copies joined with $K_1$ to distinguish it. Complete binary trees provide another example: each is built by joining two copies of a smaller tree through a common root, driving linear clique-width to infinity. (Binary trees do not contradict Theorem~\ref{thm:primes:bdd:lcw}, because the self-embedding here occurs within the prime induced subgraphs themselves, which have unbounded linear clique-width.) The same moral underlies the excluded grid theorem, which identifies large grid minors as the canonical self-similar obstruction to bounded treewidth, and the classical result of Chomsky and Sch\"utzenberger~\cite{chomsky:the-algebraic-t:} that a context-free grammar with no recursive nonterminals generates a regular language.

The connection to permutation patterns is worth elaborating. The result of Brignall, Korpelainen, and Vatter~\cite{brignall:linear-clique-w:} that we generalize here was itself the graph-theoretic analogue of the enumerative results of Albert, Atkinson, and Vatter~\cite{albert:subclasses-of-t:} for separable permutations. In that translation, cographs correspond to separable permutations, quasi-threshold and co-quasi-threshold graphs correspond to the classes $\Av(132)$, $\Av(213)$, $\Av(231)$, and $\Av(312)$, and modular decomposition corresponds to substitution decomposition. Albert, Atkinson, and Vatter showed that a subclass of the separable permutations avoiding one of these four classes has a rational generating function, while our Theorem~\ref{thm:primes:bdd:lcw} extends the graph side of this analogy beyond cographs to arbitrary hereditary classes whose prime graphs have bounded linear clique-width. Our results here suggest it may be possible to extend the results of Albert, Atkinson, and Vatter~\cite{albert:subclasses-of-t:} to the context of inflations of geometric grid classes~\cite{albert:inflations-of-g:,albert:geometric-grid-:}.

\minisec{Acknowledgements}
This research was conducted at the Oberwolfach Mini-Workshop on Permutation Patterns, \#2405c, held January 28--February 2, 2024.

\setlength{\bibsep}{4pt}

\bibliographystyle{acm}
\bibliography{lcw-refs}

@incollection{chomsky:the-algebraic-t:,
	address = {Amsterdam, The Netherlands},
	author = {Chomsky, N. and Sch{{\"{u}}}tzenberger, Marcel Paul},
	booktitle = {Computer programming and formal systems},
	doi = {10.1016/S0049-237X(08)72023-8},
	mrclass = {94.50 (68.00)},
	mrnumber = {MR0152391 (27 \#2371)},
	pages = {118--161},
	publisher = {North-Holland},
	title = {The algebraic theory of context-free languages},
	year = {1963}}

@article{wanke:k-nlc-graphs-an:,
	author = {Wanke, Egon},
	doi = {10.1016/0166-218X(94)90026-4},
	journal = {Discrete Appl. Math.},
	number = {2-3},
	pages = {251--266},
	title = {{$k$}-{NLC} graphs and polynomial algorithms},
	volume = {54},
	year = {1994}}

@article{lozin:the-relative-cl:,
	author = {Lozin, Vadim and Rautenbach, Dieter},
	doi = {10.1016/j.jctb.2007.04.001},
	journal = {J. Combin. Theory Ser. B},
	number = {5},
	pages = {846--858},
	title = {The relative clique-width of a graph},
	volume = {97},
	year = {2007}}

@article{courcelle:upper-bounds-to:,
	author = {Courcelle, Bruno and Olariu, Stephan},
	doi = {10.1016/S0166-218X(99)00184-5},
	journal = {Discrete Appl. Math.},
	number = {1-3},
	pages = {77--114},
	title = {Upper bounds to the clique width of graphs},
	volume = {101},
	year = {2000}}

@article{courcelle:handle-rewritin:,
	author = {Courcelle, Bruno and Engelfriet, Joost and Rozenberg, Grzegorz},
	doi = {10.1016/0022-0000(93)90004-G},
	journal = {J. Comput. System Sci.},
	number = {2},
	pages = {218--270},
	title = {Handle-rewriting hypergraph grammars},
	volume = {46},
	year = {1993}}

@article{albert:inflations-of-g:,
	author = {Albert, Michael Henry and Ru{\v{s}}kuc, Nikola and Vatter, Vincent},
	doi = {10.1007/s11856-014-1098-8},
	journal = {Israel J. Math.},
	number = {1},
	pages = {73--108},
	title = {Inflations of geometric grid classes of permutations},
	volume = {205},
	year = {2015}}

@article{albert:geometric-grid-:,
	author = {Albert, Michael Henry and Atkinson, Michael David and Bouvel, Mathilde and Ru{\v{s}}kuc, Nikola and Vatter, Vincent},
	doi = {10.1090/S0002-9947-2013-05804-7},
	journal = {Trans. Amer. Math. Soc.},
	number = {11},
	pages = {5859--5881},
	title = {Geometric grid classes of permutations},
	volume = {365},
	year = {2013}}

@article{albert:subclasses-of-t:,
	author = {Albert, Michael Henry and Atkinson, Michael David and Vatter, Vincent},
	doi = {10.1112/blms/bdr022},
	journal = {Bull. Lond. Math. Soc.},
	number = {5},
	pages = {859--870},
	title = {Subclasses of the separable permutations},
	volume = {43},
	year = {2011}}

@article{gurski:on-the-relation:,
	author = {Gurski, Frank and Wanke, Egon},
	doi = {10.1016/j.tcs.2005.05.018},
	journal = {Theoret. Comput. Sci.},
	number = {{{1-2}}},
	pages = {76--89},
	title = {On the relationship between {NLC}-width and linear {NLC}-width},
	volume = {347},
	year = {2005}}

@article{daligault:well-quasi-orde:,
	author = {Daligault, Jean and Rao, Michael and Thomass\'e, St\'ephan},
	doi = {10.1007/s11083-010-9174-0},
	fjournal = {Order. A Journal on the Theory of Ordered Sets and its Applications},
	issn = {0167-8094,1572-9273},
	journal = {Order},
	mrclass = {06A07 (05C12 05C78)},
	mrnumber = {2728737},
	number = {3},
	pages = {301--315},
	title = {Well-quasi-order of relabel functions},
	volume = {27},
	year = {2010}}

@article{brignall:a-framework-for:,
	author = {Brignall, R. and Cocks, D.},
	doi = {10.1137/22M1487448},
	fjournal = {SIAM Journal on Discrete Mathematics},
	issn = {0895-4801,1095-7146},
	journal = {SIAM J. Discrete Math.},
	mrclass = {05C75 (05C85)},
	mrnumber = {4660705},
	number = {4},
	pages = {2558--2584},
	title = {A framework for minimal hereditary classes of graphs of unbounded clique-width},
	volume = {37},
	year = {2023}}

@article{brignall:uncountably-man:cw,
	author = {Brignall, R. and Cocks, D.},
	doi = {10.37236/10483},
	fjournal = {Electronic Journal of Combinatorics},
	issn = {1077-8926},
	journal = {Electron. J. Combin.},
	mrclass = {05C75 (05C85)},
	mrnumber = {4406224},
	number = {1},
	pages = {Paper No. 1.63, 27 pp.},
	title = {Uncountably many minimal hereditary classes of graphs of unbounded clique-width},
	volume = {29},
	year = {2022}}

@article{atminas:minimal-classes:,
	author = {Atminas, A. and Brignall, R. and Lozin, V. and Stacho, J.},
	doi = {10.1016/j.dam.2021.02.007},
	fjournal = {Discrete Applied Mathematics. The Journal of Combinatorial Algorithms, Informatics and Computational Sciences},
	issn = {0166-218X,1872-6771},
	journal = {Discrete Appl. Math.},
	mrclass = {05C75 (05C69 05C85 06A07)},
	mrnumber = {4225694},
	pages = {57--69},
	title = {Minimal classes of graphs of unbounded clique-width defined by finitely many forbidden induced subgraphs},
	volume = {295},
	year = {2021}}

@article{collins:infinitely-many:,
	author = {Collins, Andrew and Foniok, Jan and Korpelainen, Nicholas and Lozin, Vadim and Zamaraev, Viktor},
	doi = {10.1016/j.dam.2017.02.012},
	journal = {Discrete Appl. Math.},
	pages = {145--152},
	title = {Infinitely many minimal classes of graphs of unbounded clique-width},
	volume = {248},
	year = {2018}}

@article{lozin:minimal-classes:,
	author = {Lozin, Vadim},
	doi = {10.1007/s00026-011-0117-2},
	fjournal = {Annals of Combinatorics},
	issn = {0218-0006},
	journal = {Ann. Comb.},
	mrclass = {05C75},
	mrnumber = {2854788},
	number = {4},
	pages = {707--722},
	title = {Minimal classes of graphs of unbounded clique-width},
	volume = {15},
	year = {2011}}

@article{alecu:between-clique-:,
	author = {Alecu, Bogdan and Kant\'{e}, Mamadou Moustapha and Lozin, Vadim and Zamaraev, Viktor},
	doi = {10.1016/j.disc.2020.111926},
	fjournal = {Discrete Mathematics},
	issn = {0012-365X,1872-681X},
	journal = {Discrete Math.},
	mrclass = {05C75 (05C69 05C85)},
	mrnumber = {4085903},
	number = {8},
	pages = {Article 111926, 14 pp.},
	title = {Between clique-width and linear clique-width of bipartite graphs},
	volume = {343},
	year = {2020}}

@incollection{gallai:a-translation-o:,
	address = {Chichester, England},
	author = {Maffray, Fr{\'e}d{\'e}ric and Preissmann, Myriam},
	booktitle = {Perfect Graphs},
	editor = {{Ram{\'\i}rez Alfons{\'\i }n}, Jorge Luis and Reed, Bruce Alan},
	pages = {25--66},
	publisher = {Wiley},
	series = {Wiley Series in Discrete Math. \& Optim.},
	title = {A translation of {G}allai's paper: ``{T}ransitiv orientierbare {G}raphen''},
	volume = {44},
	year = {2001}}

@article{gallai:transitiv-orien:,
	author = {Gallai, Tibor},
	doi = {10.1007/BF02020961},
	fjournal = {Acta Mathematica Academiae Scientiarum Hungaricae},
	journal = {Acta Math. Acad. Sci. Hungar.},
	number = {1-2},
	pages = {25--66},
	title = {Transitiv orientierbare {G}raphen},
	volume = {18},
	year = {1967}}

@article{ferguson:letter-graphs-a:,
	author = {Ferguson, Robert and Vatter, Vincent},
	doi = {10.1016/j.dam.2021.11.007},
	journal = {Discrete Appl. Math.},
	pages = {215--220},
	title = {Letter graphs and modular decomposition},
	volume = {309},
	year = {2022}}

@article{brignall:linear-clique-w:,
	author = {Brignall, Robert and Korpelainen, Nicholas and Vatter, Vincent},
	doi = {10.1002/jgt.22037},
	journal = {J. Graph Theory},
	pages = {501--511},
	title = {Linear clique-width for hereditary classes of cographs},
	volume = {84},
	year = {2017}}

\end{document}